\documentclass[10pt,reqno,twoside]{amsart}

\usepackage{amsmath,amsfonts,euscript,amscd,amsthm,amssymb,upref,graphics,color,verbatim, float, placeins,mathrsfs, mathtools,comment,tikz, multirow, makecell, tikz-cd}

\usetikzlibrary{positioning}

\newcommand{\Cscr}{\mathscr{C}}

\usepackage{ulem}  

\usepackage[margin=1in]{geometry}
\usepackage[all]{xy}
\usepackage{hyperref}
\usepackage[shortlabels]{enumitem}

\newcommand{\rk}{\operatorname{{\rm rk}}}

\newcommand{\Integ}{\ensuremath{\mathbb{Z}}}
\newcommand{\Nat}{\ensuremath{\mathbb{N}}}
\newcommand{\Rat}{\ensuremath{\mathbb{Q}}}
\newcommand{\Comp}{\ensuremath{\mathbb{C}}}

\newcommand{\aff}{\ensuremath{\mathbb{A}}}
\newcommand{\bk}{{\ensuremath{\rm \bf k}}}

\newcommand{\G}{\ensuremath{\mathbb{G}}}

\newcommand{\lb}{\langle}
\newcommand{\rb}{\rangle}

\newcommand{\lcm}{		\operatorname{{\rm lcm}}}

\newcommand{\setspec}[2]{\big\{\,#1\, \mid \,#2\, \big\}}
\renewcommand{\emptyset}{\varnothing}
\newcommand{\isom}{\cong}
\newcommand{\Sing}{		\operatorname{{\rm Sing}}}
\newcommand{\PPP}{\mathbb{P}}

\newcommand{\rank}{		\operatorname{\rm rank}}

\newcommand{\Spec}{		\operatorname{{\rm Spec}}}
\newcommand{\Proj}{		\operatorname{{\rm Proj}}}

\newcommand{\Supp}{		\operatorname{{\rm Supp}}}

\newcommand{\OSheaf}{\operatorname{\mathcal O}}

\swapnumbers

\newtheorem{theorem}[subsection]{Theorem}

\newtheorem*{theorem*}{Theorem}
\newtheorem{proposition}[subsection]{Proposition}
\newtheorem*{proposition*}{Proposition}
\newtheorem{lemma}[subsection]{Lemma}

\newtheorem*{lemma*}{Lemma}
\newtheorem{corollary}[subsection]{Corollary}

\theoremstyle{definition}

\newtheorem{remark}[subsection]{Remark}
\newtheorem{definition}[subsection]{Definition}

\newtheorem{nothing}[subsection]{}
\newtheorem{nothing*}[subsection]{}
\newtheorem{example}[subsection]{Example}
\newtheorem{question}[subsection]{Question}
\newtheorem*{question*}{Question}
\newtheorem{assumption}[subsection]{Assumption}
\newtheorem{notation}[subsection]{Notation}
\newtheorem*{notation*}{Notation}
\newtheorem{remarks}[subsection]{Remarks}

\newtheorem*{maintheorem}{Main Theorem}

\newcommand{\codim}{		\operatorname{{\rm codim}}}

\newcommand{\mgoth}{\mathfrak{m}}

\newcommand{\Eeul}{\EuScript{E}}

\newcommand{\Geul}{\EuScript{G}}
\newcommand{\Heul}{\EuScript{H}}

\newcommand{\Keul}{\EuScript{K}}

\newcommand{\Ueul}{\EuScript{U}}
\newcommand{\Veul}{\EuScript{V}}

\newcommand{\ba}{{\bf a}}
\newcommand{\bb}{{\bf b}}

\newcommand{\Cl}{		\operatorname{{\rm Cl}}}

\newcommand{\ML}{\operatorname{{\rm ML}}}

\newcommand{\Aff}{\mathbb{A}}

\title{The isomorphism classes of the surfaces $x_1^{a_1} + x_2^{a_2} + x_3^{a_3} + 1 = 0$}

\author[M. Chitayat]{Michael Chitayat}
\address{University of Padova, Department of Mathematics,  Via Trieste 63, 35131, Padova, PD, Italy}
\email{michael.chitayat@unipd.it}

\author[B. Hajra]{Buddhadev Hajra}
\address{Indian Statistical Institute Kolkata, Stat-Math Unit, 203 B. T. Road, Baranagar, Kolkata 700108, India}
\email{hajrabuddhadev92@gmail.com}
\date{\today}
 

\begin{document}

\begin{abstract}
Let $f = x_1^{a_1} + x_2^{a_2} + x_3^{a_3} + 1 \in \Comp[x_1,x_2,x_3]$ and let $g =  y_1^{b_1} + y_2^{b_2} + y_3^{b_3} + 1 \in \Comp[y_1,y_2,y_3]$ where $a_1,a_2,a_3,b_1,b_2,b_3 \geq 2$. We prove that the surfaces $V(f) \subset \Aff^3$ and $V(g) \subset \Aff^3$ are isomorphic if and only if $(a_1,a_2,a_3) = (b_1,b_2,b_3)$ up to a permutation of the entries.
\end{abstract}

\maketitle
\newcommand{\I}{\text{\rm I}}
\newcommand{\II}{\text{\rm II}}
\newcommand{\plinth}{\operatorname{{\rm plinth}}}

\newcommand{\bbD}{\mathbb{D}}
\newcommand{\dgoth}{\mathfrak{d}}
\newcommand{\Hom}{\operatorname{{\rm Hom}}}
\newcommand{\Span}{\operatorname{{\rm Span}}}
\newcommand{\kk}[1]{\bk^{[{#1}]}}
\newcommand{\Norm}{\operatorname{{N}}}
\newcommand{\ad}{	\operatorname{\text{\rm ad}}}

\newcommand{\UnitalAlg}[1]{\langle #1 \rangle_{\text{\rm u-alg}}}
\newcommand{\AssAlg}[1]{\langle #1 \rangle_{\text{\rm alg}}}
\newcommand{\LieAlg}[1]{\langle #1 \rangle_{\text{\rm Lie}}}

\newcommand{\lfd}{\operatorname{{\rm LFD}}}
\newcommand{\LieWLF}{\operatorname{{\rm LieWLF}}}
\newcommand{\SolLieWLF}{\operatorname{{\rm SolLieWLF}}}

\newcommand{\rtrdeg}{	\operatorname{{\rm rtrdeg}}}
\newcommand{\grank}{	\operatorname{\rm grank}}

\newcommand{\rien}[1]{}




\vspace{-10pt}

\section*{Notation and Assumptions}
\begin{notation*}
    Throughout the text, we use the following assumptions and notation:
    \begin{itemize}
        \item The set of natural numbers $\Nat$ include zero. We write $\Nat_{\geq k}$ for the set of natural numbers greater than or equal to $k$.
        
        \item A $\bk$-\textit{variety}  is an integral separated scheme of finite type over an algebraically closed field $\bk$.  A {\it curve} (resp. {\it surface}) is a $1$-dimensional (resp. $2$-dimensional) $\bk$-variety.
        \item Given a $\bk$-variety $X$, we denote its singular locus by $\Sing(X)$. 
            \item Let $X$ be a normal variety. By \textit{divisor}, we mean a  Weil divisor of $X$. A Weil divisor is \textit{$\Rat$-Cartier} if $nD$ is Cartier for some $n \in \Nat_{\geq 1}$. It is \textit{ample} if there exists some $m\in \Nat_{\geq1}$ such that $\OSheaf_X(mD)$ is a very ample invertible sheaf on $X$.  
        \item If $X$ is a normal variety, then $K_X$ is a canonical divisor of $X$. 
        
        \item We use the symbol $\zeta_n$ to denote the primitive $n^{th}$ root of unity $e^{\frac{2\pi i}{n}}$. 
        \item Given $k \in \Nat$, we use bold-faced letters $\ba = (a_1, \dots, a_k)$ or $\bb = (b_1, \dots, b_k)$ to denote a tuple in $\Nat^k$. We write $\ba \sim \bb$ if the coordinates of $\bb$ permute those of $\ba$.

    \end{itemize}
\end{notation*}

     \section{Introduction} The classification of algebraic varieties up to isomorphism is a central and often hard problem in algebraic geometry, even for families of varieties that are very special in nature. A classical illustration is provided by the family of Danielewski surfaces $\{D_n\}_{n \in \mathbb{N}}$, where
\[
D_n = V(x_1^n x_2 + x_3^2 + 1) \subset \mathbb{A}_\Comp^3.
\]
A natural question is:

\begin{question*}
Let $m,n \in \mathbb{N}$. Does $D_m \isom D_n$ imply $m = n$?
\end{question*}
The answer to this question is not immediately obvious, and yet this question has an affirmative answer. The surfaces $D_n$ are pairwise non-isomorphic and can be distinguished by their homology groups at infinity \cite{Fieseler1994}. Such examples highlight that even in seemingly simple families of affine hypersurfaces that are parametrized by $\Nat$, it may well be difficult to distinguish which members are part of the same isomorphism class. A classification of the isomorphism classes of Danielewski surfaces was later given in \cite{poloni2011classification}. While the above question was very natural in view of the importance of these Danielewski surfaces as counterexamples to the Generalized Zariski Cancellation Problem, a similar question can be asked for various interesting families of affine varieties.

In this article, we investigate analogous questions for a family of affine hypersurfaces parameterized by $(\Nat_{\geq 2})^n$. For $\ba = (a_1,\dots,a_n) \in (\mathbb{N}_{\geq 2})^n$, consider the polynomial
\[
f_{\ba} = x_1^{a_1} + \cdots + x_n^{a_n} \in \Comp[x_1,\dots,x_n],
\]
and the surjective morphism of $\Comp$-varieties 
\[
\varphi_{\ba} : \mathbb{A}^n_\Comp \longrightarrow \mathbb{A}^1_\Comp
\]
induced by the $\Comp$-algebra monomorphism $\Comp[f_\ba] \hookrightarrow \Comp[x_1, \dots, x_n].$ 

We view $\varphi_{\ba}$ as an affine fibration. A basic but crucial observation is that the restriction of $\varphi_{\ba}$ over $\mathbb{A}^1_\Comp \setminus \{0\}$ is isotrivial: for any $t \in \mathbb{A}^1_\Comp \setminus \{0\}$, the fiber $\varphi_{\ba}^{-1}(t)$ is isomorphic to $\varphi_{\ba}^{-1}(-1)$. This follows from the weighted homogeneity of $f_{\ba}$, which allows one to rescale coordinates and identify all nonzero fibers via $\Comp$-isomorphisms. Thus the affine fibration $\varphi_{\ba}$ is governed by two distinguished fibres:
\begin{itemize}
\item the special fiber $V(f_{\ba}) = \varphi_{\ba}^{-1}(0)$, which is a singular hypersurface with an isolated singularity at the origin---famously known as a {\it Pham-Brieskorn hypersurface};
\item the general fiber $V(f_{\ba}+1) = \varphi_{\ba}^{-1}(-1)$, which is a smooth hypersurface.
\end{itemize}
This leads naturally to the problem of understanding to what extent the isomorphism type of these fibers determines $\ba$. 

Given $\ba, \bb \in (\mathbb{N}_{\geq 2})^n$, we write $\ba \sim \bb$ if $\bb$ is obtained from $\ba$ by a permutation of its entries. Consider the following two questions (the $n = 3$ case of Question (2) was asked by R. V. Gurjar):

\begin{question*}
Let $\ba, \bb \in (\mathbb{N}_{\geq 1})^n$, and consider the associated affine fibrations $\varphi_{\ba}, \varphi_{\bb} : \mathbb{A}^n_\Comp \to \mathbb{A}^1_\Comp$ defined above.
\begin{enumerate}[\rm(1)]
\item Suppose $V(f_{\ba}) \isom V(f_{\bb})$ (the special fibers are isomorphic). Is $\ba \sim \bb$?
\item Suppose $V(f_{\ba} + 1) \isom V(f_{\bb} + 1)$ (the general fibers are isomorphic). Is $\ba \sim \bb$?
\end{enumerate}
\end{question*}

Question (1) is natural in view of the fact that Pham-Brieskorn hypersurfaces are widely studied from various perspectives. (A simple internet search will yield hundreds of academic articles on the varieties and manifolds defined by these equations.) Question (1) asks whether the isomorphism class of a Pham-Brieskorn hypersurface determines its exponents, while Question (2) asks whether the same is true for smooth affine hypersurface $V(f_{\ba}+1)$.  Although all fibers of $\varphi_{\ba}$ over $\mathbb{A}^1_\Comp\setminus \{0\}$ for a fixed $\ba$ are isomorphic, it is not immediately clear whether their common isomorphism type determines $\ba$.

Note that the seemingly subtle distinction between Questions (1) and (2) plays an essential role in our analysis. While Question (1) has an affirmative answer for all $n$ (see Proposition \ref{PBexponents}), Question (2) presents more of a difficulty. The main issue is that the hypersurfaces $V(f_{\ba}+1)$ and $V(f_{\bb}+1)$ are smooth, and hence lack the singular structure that provides the accessible invariants used to resolve Question (1).

Question (2) is trivial in the $n=1$ case, and was settled for $n=2$ (i.e., for curves) in \cite{gurjar2025classification}. The main result of this article is an affirmative answer to Question (2) in the $n=3$ case.

\begin{maintheorem}\label{mainTheorem}
Let $\ba, \bb \in (\mathbb{N}_{\geq 2})^3$ and let $f_{\ba}, f_{\bb} \in \mathbb{C}[x_1,x_2,x_3]$. If $V(f_{\ba} + 1) \isom V(f_{\bb} + 1)$ then $\ba \sim \bb$.
\end{maintheorem}

We emphasize that, for fixed tuples $\ba,\bb \in (\mathbb{N}_{\geq 2})^k$, it is often possible to find an invariant that distinguishes the hypersurfaces $V(f_{\ba}+1)$ and $V(f_{\bb}+1)$. The difficulty lies in proving that such a distinguishing invariant must exist {\it a priori}, without any assumptions on the specific chosen tuples $\ba$ and $\bb$. The main result of this article shows that this is indeed the case for the 2-dimensional hypersurfaces of form $V(f_\ba + 1)$.

Let us give the main ideas behind our proof of the Main Theorem. We let $\ba, \bb \in (\Nat_{\geq 2})^3$ and define $X_\ba = V(f_\ba +1) = V(x_1^{a_1} + x_2^{a_2} + x_3^{a_3} + 1) \subset \Aff_\Comp^3$ and $X_\bb = V(f_\bb + 1) =  V(x_1^{b_1} + x_2^{b_2} + x_3^{b_3} + 1) \subset \Aff_\Comp^3$. Our goal is to show that if $X_\ba \isom X_\bb$ then $\ba\sim \bb$.

    The first step is to partition $(\Nat_{\geq 2})^3$ into four subsets $S_0, S_{+,+}, S_{+,0}, S_{+,-}$ with the property that if $X_\ba \isom X_\bb$, then both $\ba$ and $\bb$ belong to the same subset. We call subsets with this property \textit{parameter-closed subsets} (see Section \ref{paramClosed}). The fact that these four subsets are parameter-closed is established in Lemma \ref{theEasyPart}. The proof of that lemma mainly depends on two things:
    \begin{enumerate}
        \item that $X_\ba$ can always be embedded as an affine open subset of a well-formed quasismooth weighted projective surface,  which we denote by $\overline{X_\ba}$ (see Section \ref{section:Setup});
        \item the birational geometry of projective surfaces with rational singularities (see Propositions \ref{IsoGeneralization}, \ref{mapExtends}).
    \end{enumerate} 
    The second step is to show that if $\ba$ and $\bb$ are in any of the four subsets above, then $X_\bb \isom X_\ba$ implies $\bb \sim \ba$. When $\ba, \bb \in S_{+,+}$ we prove that any isomorphism between $X_\ba$ and $X_\bb$ extends to an isomorphism between $\overline{X_\ba}$ and $\overline{X_\bb}$. This allows us to use a recent result of Esser \cite[Theorem 2.1]{esser2024automorphisms} to conclude that $\ba \sim \bb$. It turns out that $S_{+,0}$ is a  finite set. If $\ba, \bb \in S_{+,0}$ and $X_\ba \isom X_\bb$, the fact that $\ba \sim \bb$ can be deduced by computing two invariants (the rank of the divisor class group and the genus of the boundary curve) for each surface $X_\ba$ such that $\ba \in S_{+,0}$. If $\ba, \bb \in S_{+,-}$ and $X_\ba \isom X_\bb$, the equivalence $\ba \sim \bb$ is deduced by combining a recent result of Daigle on locally nilpotent derivations \cite[Proposition 10.5]{daigle2023rigidity} together with the same two invariants used in the $S_{+,0}$ case. Finally, if $\ba, \bb \in S_0$, the proof combines:
    \begin{itemize}
        \item a characterization of when the surfaces $V(f_\ba) \subseteq \Aff^3$ are rational (see Proposition \ref{S0char}),
        \item an analysis of the singular points of the weighted projective surface $\overline{X_\ba}$ (see Section \ref{Section:Singularities}),
        \item an analysis of weighted graphs that appear as dual graphs of SNC boundaries (see \ref{subsec:weightedGraphs} to \ref{minimalGraph3branches}).
    \end{itemize}
    
    \medskip 
    The article is organized as follows. We give the necessary preliminaries in Section \ref{Sec:Preliminaries}. Section \ref{section:Setup} establishes the notation, Section \ref{Section:Singularities} classifies the singularities of certain weighted projective hypersurfaces, Section \ref{paramClosed} defines and gives basic properties of parameter-closed subsets and Section \ref{proofOfTheorem} proves the Main Theorem using the previously established results.

    \medskip
    \noindent \textbf{Acknowledgments.} Both authors would like to thank: Neena Gupta and Adrien Dubouloz for the invitation to speak at the Indo-European Conference on Mathematics (IECM-2026), jointly organized by the European Mathematical Society and the Indian Mathematics Consortium, held in Pune, India, where this collaboration started; Daniel Daigle for useful discussions around weighted graphs and for sharing the reference \cite{daigle2001weighted}; R.V. Gurjar for asking the central question of this paper.
    
    The second-named author expresses his sincere gratitude to R. V. Gurjar for his guidance throughout this project. He is also grateful to Neena Gupta for useful discussions about Makar-Limanov invariants.
    
\medskip
    \noindent \textbf{Funding.} The first-named author's travel to Pune was partially funded by the ``Algebraic groups and their actions" project at the University of Padova, with project code 2025DM1DOR25-00862. The second-named author acknowledges financial support from the Department of Science and Technology, Government of India, through the INSPIRE Faculty Fellowship (Reference No.: DST/INSPIRE/04/2024/003379).

\medskip

\noindent \textbf{Statements and Declarations}. The authors declare that they have no conflicts of interest related to this manuscript. No data were generated or used in the course of this study.

\section{Preliminaries}\label{Sec:Preliminaries}

The following proposition gives an affirmative answer to Question (1), valid in any dimension. While this result may be well known, we include a proof due to the second author and M. Upmanyu. A more general version can be found in \cite{hajra2026generalizedzariskicancellationbrieskornpham}.

\begin{proposition}\label{PBexponents}
    Let $f_\ba =x_1^{a_1} + \dots + x_n^{a_n}$ and let $f_\bb = x_1^{b_1} + \dots + x_n^{b_n}$ where $f_\ba, f_\bb \in \Comp[x_1, \dots, x_n]$. If $V(f_\ba) \isom V(f_\bb)$, then $\ba \sim \bb$. 
\end{proposition}
\begin{proof}
    Let $M_\ba$ and $T_\ba$ respectively denote the Milnor and Tjurina algebras of the singularity of $V(f_\ba)$. Since $f_\ba \in \lb \{ x_i^{a_i-1}\}_{i=1}^n \rb$, $M_\ba = T_\ba$. We obtain
    $$\frac{\Comp\{x_1,\ldots,x_n\}}{\lb x_1^{a_1-1},\ldots,x_n^{a_n-1} \rb} =  M_\ba \isom T_\ba \isom T_\bb \isom M_\bb = \frac{\Comp\{y_1,\ldots,y_n\}}{\lb y_1^{b_1-1},\ldots,y_n^{b_n-1} \rb}$$
    where the middle isomorphism is by \cite[Theorem (i) $\Rightarrow$ (ii)]{MY1982}. Since $(M_\ba, \mgoth_\ba)$ and $(M_\bb, \mgoth_\bb)$ are local rings, any isomorphism must map $\mgoth_\ba$ to $
    \mgoth_\bb$. Consequently, the Hilbert polynomials of $M_\ba$ and $M_\bb$ are equal. Comparing them gives $$\prod\limits_{i=1}^{n}(1+t+t^2+\cdots+t^{a_i-2}) = \prod\limits_{i=1}^{n}(1+t+t^2+\cdots+t^{b_i-2}). $$
    Factoring these two polynomials in $\bar{\Rat}[t]$ into products of linear terms, we deduce that $\underset{1 \leq i \leq n}{\max}\{a_i - 2\} = \underset{1 \leq i \leq n}{\max}\{b_i - 2\}$. Assume without loss of generality that this maximum is achieved at $a_n - 2 = b_n - 2$. Then $a_n = b_n$ and we can divide both sides by $1 + t + \cdots + t^{a_n - 2}$. Repeating this process gives $\ba \sim \bb$.

\end{proof}

\noindent{\it Weighted graphs and their equivalence classes.} 
\begin{nothing}
\label{subsec:weightedGraphs}
    Throughout this article, a \textit{weighted graph} $\Geul$ is a finite connected undirected graph such that 
    \begin{itemize}
        \item no edge connects a vertex to itself
        \item there is at most one edge between each pair of vertices
        \item each vertex is decorated with an integer $n \in \Integ$ call a \textit{weight}.
    \end{itemize} 

Let $\Geul$ be a weighted graph.
\begin{enumerate}[\rm(a)]

\item Let $v$ be a vertex of $\Geul$.
The {\it blowing-up of $\Geul$ at $v$} is the weighted graph obtained from $\Geul$ by adding a vertex $x$ of weight $(-1)$,
adding the edge $\{v,x\}$, and decreasing the weight of $v$ by $1$.

\item Let $\epsilon = \{u,v\}$ be an edge of $\Geul$.
The {\it blowing-up of $\Geul$ at $\epsilon$} is the weighted graph obtained from $\Geul$ by deleting the edge $\epsilon$,
adding a vertex $x$ of weight $(-1)$ and two edges $\{u,x\}$ and $\{v,x\}$,
and decreasing the weights of $u$ and $v$ by $1$.

\end{enumerate}

In (a) and (b), we refer to $x$ as {\it the vertex created by the blowing-up.}

\medskip
The \textit{degree} of a vertex $x$ is the number of edges incident to $x$. A vertex $x$ of a weighted graph $\Heul$ is {\it superfluous\/} if
its weight is $(-1)$,
its degree is $1$ or $2$,
and no two neighbors of $x$ are neighbors of each other.
It is not hard to see that a vertex $x$ of $\Heul$ is superfluous if and only if $\Heul$ is the result of blowing-up some 
weighted graph $\Geul$ at some vertex or edge in such a way that $x$ is the vertex created by that blowing-up.
The graph $\Geul$ is then uniquely determined by the pair $(\Heul,x)$, and is called the {\it blowing-down of $\Heul$ at $x$}.
A weighted graph is {\it minimal\/} if it does not contain any superfluous vertex, or equivalently, if it cannot be blown down.
Weighted graphs $\Geul_1, \Geul_2$ are \textit{equivalent} (and we write $\Geul_1 \sim \Geul_2$) if $\Geul_1$ and $\Geul_2$ can be obtained from one another via a sequence of blowing up and blowing down operations. Moreover, $\sim$ defines an equivalence relation on the set of weighted graphs. (Note that the assumption that our weighted graphs are connected allows us to avoid considering the ``free blowing up" operation which needs to be considered when studying equivalence classes of weighted graphs that are not necessarily connected. See \cite[p.65]{DaigleLinearWeighted} for example.)   
\end{nothing}

\begin{nothing}\label{weightedGraphDef}
    Let $X, Y$ be smooth affine surfaces and let $\tilde{X}$ be a smooth projective surface containing $X$ such that $\partial = \tilde{X} \setminus X$ is an SNC divisor. Since $X$ is affine, $\partial$ is connected (by \cite[Theorem 4.3]{hartshorne1972algebraic} for example). Then $\tilde{X}$ is called a \textit{log smooth compactification of $X$}. We recall the following facts:
    \begin{enumerate}[\rm(a)]
    \item if $\tilde{X}$ and $\hat{X}$ are log smooth compactifications of $X$, then $\hat{X}$ can be obtained from $\tilde{X}$ via a sequence of blowing-up and blowing-down operations; furthermore, the sequence of blowing-up and blowing-down operations can be chosen so that all such operations occur in the complement of $X$; 
    \item the multisets of non-zero genera of curves that appear on the boundaries $\tilde{X} \setminus X$ and $\hat{X} \setminus X$ are equal. 
    \end{enumerate}
    We also define the weighted graph $\Geul(X,\tilde{X})$ (which too is connected) as follows:
    \begin{itemize}
        \item to each irreducible curve $\partial_i \subseteq \Supp \partial$, assign a vertex $v_i$ whose weight $w_i$ is the self-intersection $\partial_i^2$,
        \item given $i \neq j$, connect vertices $v_i$ and $v_j$ by an (undirected) edge $(v_i,v_j)$ if and only if the curves $\partial_i$ and $\partial_j$ intersect. 
    \end{itemize}    
From (a), one obtains
\begin{equation}\label{SNCBoundaryTheorem1}
    \text{if $\tilde{X}$ and $\hat{X}$ are log smooth compactifications of $X$, then $\Geul(X,\tilde{X}) \sim \Geul(X,\hat{X})$} 
\end{equation}
and
\begin{equation}\label{SNCBoundaryTheorem2}
    \text{if $X\isom Y$ and $\tilde{X}$, $\tilde{Y}$ are log smooth compactifications of $X$ and $Y$, then $\Geul(X,\tilde{X}) \sim \Geul(Y,\tilde{Y})$}.  
\end{equation}
\end{nothing}

Our next goal is to prove Lemma \ref{minimalGraph3branches} which is used in Lemma \ref{3branches}. Before doing so we need:

\begin{lemma}\label{shorterSequence}
    Let $(\Heul_0, \dots, \Heul_m)$ be a sequence of weighted graphs such that $m \geq 1$  and such that $\Heul_{i+1}$ is a blow-up of $\Heul_i$ for each $i < m$.  Let $w$ be the vertex of $\Heul_1$ that is created by the first blow-up and view $w$ as a vertex of $\Heul_m$. Suppose that $w$ is a superfluous vertex of $\Heul_m$ and let $\Heul_{m+1}$ be the blowing-down of $\Heul_m$ at $w$.  Then there exists a sequence $(\Keul_0, \dots , \Keul_{m-1})$ of weighted graphs such that $\Keul_0 = \Heul_0$, $\Keul_{m-1} = \Heul_{m+1}$, and $\Keul_{i+1}$ is a blow-up of $\Keul_i$ for each $i < m-1$.
\end{lemma}
\begin{proof}
    We proceed by induction on $m$, the case $m=1$ being clear.
Suppose that $m>1$ and consider $(\Heul_0, \dots, \Heul_m, \Heul_{m+1})$ as in the statement.
Define $\xi_0$ and $\xi_1$ by declaring that, for each $i \in \{0, 1\}$, $\Heul_{i+1}$ is the blowing-up of $\Heul_i$ at $\xi_i$,
where $\xi_i$ is either a vertex or an edge of $\Heul_i$.
If $\xi_1$ is the vertex $w$ or an edge of $\Heul_1$ incident to $w$ then the weight of $w$ in $\Heul_2$ is strictly less than $-1$;
this implies that the weight of $w$ in $\Heul_m$ is strictly less than $-1$, which contradicts the hypothesis that $w$ is a superfluous vertex of $\Heul_m$.
So
\begin{equation*}
\text{$\xi_1$ is not $w$ and $\xi_1$ is not an edge of $\Heul_1$ incident to $w$.}
\end{equation*}
It follows that the first two blowups of the sequence $\Heul_0 , \Heul_1 , \dots,  \Heul_m$ can be performed in the opposite order.
In other words,
\begin{itemize}

\item $\xi_1$ is a vertex or edge of $\Heul_0$, so it makes sense to define the weighted graph $\Keul_1$ to be the blowing-up of $\Heul_0$ at $\xi_1$;

\item $\xi_0$ is a vertex or edge of $\Keul_1$, and the blowing-up of $\Keul_1$ at $\xi_0$ is $\Heul_2$.

\end{itemize}
Also define $\Keul_0 = \Heul_0$.
The situation is represented schematically by the following picture:
$$
\xymatrix{
\Heul_1 \ar[d]_-{\xi_0} & \ar[l]_-{\xi_1} \Heul_2 \ar[d]^-{\xi_0} \\
\Heul_0 = \Keul_0 & \ar[l]^-{\xi_1} \Keul_1 
}
$$
where we write $\Ueul \xleftarrow{ \ \xi \ } \Veul$ to indicate that $\Veul$ is the blowing-up of
the weighted graph $\Ueul$ at $\xi$, where $\xi$ is either a vertex or an edge of $\Ueul$.
Moreover, the vertex of $\Heul_2$ created by the blowing-up $\Keul_1 \xleftarrow{ \ \xi_0 \ } \Heul_2$ is $w$.
So $(\Keul_1, \Heul_2, \dots, \Heul_m, \Heul_{m+1})$ satisfies the hypothesis of the Lemma and is shorter than $(\Heul_0, \dots, \Heul_m, \Heul_{m+1})$.
By the inductive hypothesis, there exist $\Keul_2,\dots,\Keul_{m-1}$ such that
$\Keul_{m-1} = \Heul_{m+1}$ and $\Keul_{i+1}$ is a blowup of $\Keul_i$ for each $i < m-1$.
Then $(\Keul_0,\dots,\Keul_{m-1})$ satisfies the requirements of the Lemma.
\end{proof}

\begin{lemma}\label{minimalGraph3branches}
    Let $\Geul$ be a star-shaped weighted graph whose central vertex $v_0$ has degree at least 3. If every vertex of $\Geul$ other than $v_0$ has weight at most $-2$, then $\Geul$ is the unique minimal graph in its equivalence class. 
\end{lemma}
\begin{proof}
Let $\Geul’$ be any minimal element of the equivalence class of $\Geul$. Then there exists a sequence of weighted graphs
\begin{equation}
    \Geul_0, \Geul_1, \dots, \Geul_n
\end{equation}
such that $\Geul_0 = \Geul, \Geul_n = \Geul'$ and for each $i = 0, \dots, n-1$, $\Geul_{i+1}$ is either a blowing-up or a blowing-down of $\Geul_i$. Of all such sequences, fix $\Geul_0, \dots , \Geul_n$ as the shortest one. We claim that $n=0$. Arguing by contradiction, assume that $n > 0$. Observe that if $\Geul_n$ is a blowing-up of $\Geul_{n-1}$ then $\Geul_n = \Geul'$ is not minimal, contradicting the assumption that $\Geul'$ is minimal. So $\Geul_n$ is a blowing-down of $\Geul_{n-1}$.  In particular, the set $D = \setspec{ i}{\Geul_{i+1} \text{ is a blowing down of $\Geul_{i}$}}$ is non-empty. Let $j$ be the least element of $D$, so that $(\Geul_0, \dots, \Geul_j)$ is a sequence of blowing-ups and $\Geul_{j+1}$ is a blowing-down of $\Geul_j$.  Consider a vertex $v$ of $\Geul_0$; then $v$ exists in $\Geul_j$, 
\begin{equation}\label{graph1}
    \text{$\deg(v)$ in $\Geul_j$  is greater than or equal to $\deg(v)$ in $\Geul_0$}
\end{equation}
and     
\begin{equation}\label{graph2}
    \text{the weight of $v$ in $\Geul_j$  is less than or equal to the weight of $v$ in $\Geul_0$.}
\end{equation}
Since every non-central vertex of $\Geul_0$ has weight at most $-2$ and the central vertex of $\Geul_0$ has degree at least 3, \eqref{graph1} and \eqref{graph2} imply that $v$ cannot be a superfluous vertex of $\Geul_j$. Now, $\Geul_{j+1}$ is the blowing-down of $\Geul_j$ at some superfluous vertex $w$; furthermore, $w$ is not a vertex that was present in $\Geul_0$, so $w$ was created by one of the blowings-up in the sequence $(\Geul_0, \dots , \Geul_j)$. That is, there exists $i  \leq j$ such that $w$ is the vertex created by the blowing-up $(\Geul_{i-1}, \Geul_i)$. 
Applying Lemma \ref{shorterSequence} to the sequence $(\Heul_0, \dots, \Heul_m, \Heul_{m+1}) = (\Geul_{i-1}, \Geul_i, \dots , \Geul_j, \Geul_{j+1})$ shows that $\Geul_{j+1}$ can be obtained from $\Geul_{i-1}$ by performing a sequence of $(j-i)$ blowups, so 2 fewer operations than in $(\Geul_{i-1}, \Geul_i, \dots, \Geul_j, \Geul_{j+1})$.  This contradicts the minimality of the sequence $(\Geul_0, \dots, \Geul_n)$.
\end{proof}

\noindent {\it Continued Fractions and Cyclic Quotient Singularities.}

\begin{nothing}\label{rationalSingularities}\cite{hirzebruch1953vierdimensionale}, \cite{daigle2001weighted} 
    It is well known that cyclic quotient singularities are rational and that the exceptional locus in a resolution of a rational surface singularity is a tree of rational curves \cite{ArtinIsolatedSingularities}. The configuration of these curves and their self-intersections can be computed as follows. 

    Given $a_1, \dots, a_r \in \Nat_{\geq 1}$, $[a_1,a_2,\dots, a_r]$ denotes the Hirzebruch-Jung continued fraction 
    $$a_1 - \cfrac{1}{a_2 - \cfrac{1}{a_3 - \cfrac{1}{ \cdots - \cfrac{1}{a_r}}}}.$$
    We note that every positive rational number $\frac{n}{m}$ admits a unique representation as a Hirzebuch-Jung continued fraction.  

    Recall that a $\frac{1}{n}(m,1)$ singularity is defined to be the singularity at the origin $p$ of the variety $X = \Spec(\Comp[x,y]^G)$ where $G$ is the  action of $\Integ / n\Integ$ on $\Comp[x,y]$ by $\zeta \cdot x = \zeta^m x$, $\zeta \cdot y = \zeta y$. Let $[a_1, \dots, a_r]$ denote the Hirzebruch-Jung continued fraction of $\frac{n}{m}$. Let $\pi : \widetilde{X} \to X$ denote the minimal resolution of $X$. Then $\pi^{-1}(p)$ is a chain of rational curves $E_1, \dots, E_r$ that satisfy $E_i^2 = -a_i$, $E_1$ intersects $\widetilde{V(y)}$ transversely and $E_k$ intersects $\widetilde{V(x)}$ transversely. Conversely, if $q$ is a $\frac{1}{n'}(m',1)$ singularity such that $\pi^{-1}(q)$ is a chain of rational curves $D_1, \dots, D_k$ such that $D_i^2 = E_i^2$ for all $i = 1, \dots, k$, $D_1$ intersects $\widetilde{V(y)}$ transversely and $D_k$ intersects $\widetilde{V(x)}$ transversely, then $n = n'$ and $m = m'$.
\end{nothing}

\noindent {\it Some Basic Birational Geometry.}

\begin{proposition}\label{IsoGeneralization}
Let $X$ and $Y$ be normal projective varieties over $\bk$, and suppose that $K_X$ and $K_Y$ are $\Rat$-Cartier Weil divisors. Let $f: X \dashrightarrow Y$ be a birational map that restricts to an isomorphism on a set $U \subseteq X$ such that $\codim_X(X \setminus U) \geq 2$, $\codim_Y(Y \setminus f(U)) \geq 2$.
\begin{enumerate}[\rm(a)]
    \item If $K_X$ is ample, then $K_Y$ is non-trivial. 
    \item If $K_X$ and $K_Y$ are ample, then $X$ and $Y$ are isomorphic.
\end{enumerate}

\end{proposition}
\begin{proof}
The indeterminacy locus $Z_X \subseteq X \setminus U$ satisfies $\codim_X(Z_X) 
\geq 2$. Let $V = f(U)$, let $Z_Y$ denote the indeterminacy locus of $f^{-1}: Y \dashrightarrow X$ and observe that $Z_Y \subseteq Y \setminus V$. Since $K_X$ and $K_Y$ are $\Rat$-Cartier, there exists a positive integer $r$ such that $rK_X$ and $rK_Y$ are Cartier. Since $f: U \to V$ is an isomorphism of open sets, for each $m \geq 0$, we have a canonical isomorphism $f^*\OSheaf_{V}(mrK_V) \cong \OSheaf_U(mrK_U).$
Since $\codim_X(X\setminus U) \geq 2$ and $\codim_Y(Y \setminus V) \geq 2$, we have $K_X|_U = K_U$ and $K_Y|_V = K_V$; thus we obtain a pullback map
on sections:
$$
f^*: H^0(V, \OSheaf_Y(mrK_Y)|_V) \longrightarrow H^0(U, \OSheaf_X(mrK_X)|_U).
$$
Since $Y \setminus V$ and $X \setminus U$ are sets of codimension at least two, the vertical arrows in the diagram below are isomorphisms, inducing a commutative diagram of isomorphisms for every $m \geq 0$. 
\[
\begin{tikzcd}
H^0\!\left(Y,\OSheaf_Y(mrK_Y)\right)
\arrow[r]
\arrow[d]
&
H^0\!\left(X,\OSheaf_X(mrK_X)\right)
\arrow[d]
\\
H^0\!\left(V,\OSheaf_Y(mrK_Y)|_V\right)
\arrow[r]
&
H^0\!\left(U,\OSheaf_X(mrK_X)|_U\right)
\end{tikzcd}
\]
For (a), assume $K_X$ is ample. Then $\dim_\bk H^0(X,\OSheaf_X(mrK_X))$ approaches infinity for $m >>0$. If $K_Y = 0$, then $\dim_\bk H^0(X,\OSheaf_X(mrK_X)) = \dim_\bk H^0(Y,\OSheaf_X(mrK_Y)) = 1$ for all $m$, a contradiction. 

For (b), we now assume that both $K_X$ and $K_Y$ are ample. Without loss of generality we may assume that the $r$ chosen above is such that $rK_X$ and $rK_Y$ are both very ample Cartier divisors. We then obtain an isomorphism of graded $\bk$-algebras:
$$
f^*: R(Y, K_Y)^{(r)} \overset{\sim}\rightarrow R(X, K_X)^{(r)},
$$
where $R(X, K_X)^{(r)} = \bigoplus_{m \geq 0} H^0(X, \OSheaf_X(mrK_X))$ is the $r^{th}$ Veronese subring of $\bigoplus_{m \geq 0} H^0(X, \OSheaf_X(mK_X))$. This gives
$$
Y \isom \Proj (R(Y, K_Y)^{(r)}) \cong \Proj (R(X, K_X)^{(r)}) \cong X, 
$$   
the first and last isomorphisms because $rK_X$ and $rK_Y$ are Cartier and very ample. 
\end{proof}

\begin{proposition}\label{mapExtends}
    Let $X$ and $Y$ be normal projective surfaces with at most rational singularities. Suppose $f:X \dashrightarrow Y$ is a rational map that restricts to an isomorphism $f|_{X\setminus C_1} : X \setminus C_1 \to Y \setminus C_2$ where $C_1$ and $C_2$ are irreducible smooth curves. If $p_g(C_1) > 0$ or $p_g(C_2) > 0$, then $f$ restricts to an isomorphism on its domain of definition.
\end{proposition}

\begin{proof}
    The rational map $f$ is defined everywhere except at finitely many points $Z = \{p_1, \dots p_k\}$ of $X$ that lie along $C_1$. Let $U$ denote the domain of definition of $f$. Resolving the indeterminacy of $f$ we obtain the following commutative diagram:
    \[
\begin{tikzcd}
\widetilde{X}
\arrow[d, "\pi"']
\arrow[dr, "\widetilde{f}"] 
& \\
X 
\arrow[r, dashed, "f"'] 
& Y
\end{tikzcd}
\]
where $\tilde{f}$ is a morphism that extends the rational map $f$. Because $X$ has at most rational singularities, the total transform $\pi^{-1}(C_1)$ is supported on $\widetilde{C_1} \cup \bigcup_{i = 1}^n E_i$ where $E_i \isom \PPP^1$ for all $i$ (by \ref{rationalSingularities}); furthermore since $C_1$ is smooth, $\widetilde{C_1} \isom C_1$. Now $\tilde{f}$ is a dominant morphism of projective surfaces and so it is surjective. 

Assume $p_g(C_1) > 0$. Then $\widetilde{C_1}$ cannot be contracted to a point by $\tilde{f}$ because all singular points of $Y$ are rational. So $\tilde{f}$ restricts to an isomorphism on $\tilde{f}|_{\tilde{X} \setminus \bigcup E_i} : \tilde{X} \setminus \bigcup E_i \to Y \setminus \tilde{f}(\bigcup E_i)$. Since $\tilde{f}$ extends the morphism $f|_U$, we have $f|_U = f|_{X \setminus Z}$ is an isomorphism. This proves the $p_g(C_1) > 0$ case.   
Assume $p_g(C_2) > 0$. By surjectivity of $\tilde{f}$, there exists a curve $C_k$ supported on $\pi^{-1}(C_1)$ whose image under $\tilde{f}$ is $C_2$. Since there is no dominant morphism from $\PPP^1$ to a curve of higher genus, $p_g(C_k)>0$ and hence $C_k = \widetilde{C_1}$. Since $\widetilde{C_1} \isom C_1$, $C_1$ has positive genus and the result follows from the case $p_g(C_1) > 0$.  

\end{proof}

\noindent {\it Quasismooth hypersurfaces.}

\begin{nothing}
Let $S = \bk_{w_0, \dots, w_n}[x_0, \dots, x_n]$ denote the graded polynomial ring where $n \geq 1$ and $\deg(x_i) = w_i$ for $i = 0, \dots, n$. The weighted projective space $\PPP = \PPP(w_0, \dots , w_n) = \Proj(S)$ is said to be \textit{well-formed} if $\gcd(w_0,\dots,w_{i-1},\hat{w_i},w_{i+1}, \dots , w_n) = 1$ for each $i = 0, \dots, n$. Every weighted projective space is a projective variety and is isomorphic to a well-formed weighted projective space. 

A \textit{weighted projective variety} $X$ is a closed subvariety of a weighted projective space. Whenever we write ``the variety $X \subseteq \PPP$", we mean that $X$ is a closed subvariety of $\PPP$. Since $\PPP(w_0, \dots , w_n)$ is a projective variety, every weighted projective variety is also a projective variety.  

Let $I$ be a homogeneous prime ideal of the graded ring $S = \Comp_{w_0,\dots,w_n}[x_0, \dots, x_n]$
	and consider the closed subvariety $X = V_+(I)$ of $\PPP=\PPP(w_0,\dots,w_n) = \Proj(S)$.  Note that $X$ is isomorphic to $\Proj( S/I )$.
	The closed subset $C_X = V(I) \subseteq \aff^{n+1}$ is called the  \textit{affine cone over $X$}; $C_X$ passes through the origin of $\aff^{n+1}$, and that $C_X \isom \Spec( S/I )$ is an integral affine scheme.	The variety $X$ is \textit{quasismooth} if $C_X$ is nonsingular away from the origin. In case $I = \lb f \rb \lhd S$ is principal, we say that $X = V_+(f)$ is a \textit{weighted projective hypersurface}. A \textit{linear cone} is a weighted hypersurface $V_+(f) \subseteq \PPP(w_0, \dots, w_n)$ such that $\deg(f) = w_i$ for some $i = 0, \dots, n$. A weighted projective hypersurface $X \subseteq \PPP(w_0, \dots, w_n)$ is \textit{well-formed} if the following two conditions hold:
    \begin{enumerate}[\rm(i)]
        \item $\PPP(w_0, \dots, w_n)$ is well-formed,
        \item $\codim_X(X \cap \Sing(\PPP(w_0, \dots, w_n)) \geq 2$.
    \end{enumerate}
\end{nothing}
\begin{remarks}\label{quasismoothRemark}
    Let $X = V_+(f) \subseteq \PPP(w_0, \dots, w_n) = \Proj S$ be a well-formed quasismooth weighted hypersurface of degree $d$. The integer $\alpha = d - \sum_{i = 0}^n w_i$ is called the \textit{amplitude of $X$}. \footnote{This is just a very special case of the amplitude of a quasismooth weighted complete intersection. See \cite{iano-fletcher_2000} for example.} The following properties are well-known:
    \begin{enumerate}[\rm(a)]
        \item $X$ is normal,
        \item $\Sing(X) = X \cap \Sing(\PPP)$, 
        \item $X$ has at most cyclic quotient singularities,
        \item the canonical divisor $K_X$ is ample if and only if $\alpha > 0$,  
        \item the geometric genus of $X$ is $p_g(X) = \dim_\bk H^{\dim(X)}(X, \OSheaf_X) = \dim_\bk(S / \lb f \rb)_\alpha$. 
    \end{enumerate}
    Parts (a),(d), and (e) appear in \cite{dolgachev}. Part (b) is \cite[Proposition 8]{Dimca1986} . Part (c) is given in \cite[1.3.3]{iano-fletcher_2000}. 
\end{remarks}   
    
    We require the following theorem:

\begin{theorem}[cf. {\cite[Theorem 2.1]{esser2024automorphisms}}]
\label{esser}

     Let $X \subset \PPP(w_0,\dots, w_{n+1})$ and $X' \subset \PPP(w_0', \dots, w_{n+1}')$ be two complex weighted projective hypersurfaces of weighted degrees $d$ and $d'$ respectively. Suppose further that $X$ and $X'$ are well-formed and quasismooth, neither is a linear cone, and one of the following holds:
     \begin{enumerate}[\rm(1)]
     \item $n \geq 3$ or
     \item $n = 2$ and $w_0 + w_1 + w_2 + w_3 \neq d$.
     \end{enumerate}
     Then, if $g : X' \to X$ is an isomorphism, we have $d = d'$, the $w_i$ coincide with the $w_i'$ up to
rearrangement and $g$ is induced by an automorphism of $\PPP(w_0,\dots , w_{n+1})$.
\end{theorem}

\section{Setup}\label{section:Setup}
\begin{assumption}
    From this point until the end of the article, we assume our base field is $\Comp$.
\end{assumption}

Let $\textbf{a} = (a_1, \dots, a_n), \textbf{b} = (b_1, \dots, b_n) \in (\Nat_{\geq 2})^n$. We define an equivalence relation by $\ba \sim \bb$ if the entries of $\textbf{b} = (b_1,\dots b_n)$ are a permutation of those of $\textbf{a} = (a_1,\dots, a_n)$. Given $\ba = (a_1, \dots, a_n)$, define 
\begin{equation}
    X_{\ba} = X_{a_1, \dots, a_n} = V(x_1^{a_1} + x_2^{a_2} + \dots +x_n^{a_n} + 1) \subset \Aff^n_\Comp
\end{equation}
and note that $X_\ba$ is nonsingular. Recall that we are interested in the following:

\begin{question}\label{mainQuestion}
    Let $\ba, \bb \in (\Nat_{\geq 2})^n$. If $X_\ba \isom X_\bb$, does it follow that $\ba \sim \bb$? 
\end{question}

As stated in the Introduction, we give an affirmative answer in the $n = 3$ case.

\medskip
We preserve the notation of \ref{setupNotation} until the end of the article. 
\begin{nothing}\label{setupNotation}

Given $\ba = (a_1,a_2,a_3) \in (\Nat_{\geq 2})^3$, let $L = \lcm(a_1,a_2,a_3)$ and let $w_i = L / a_i$ for each $i = 1,2,3$. The surface $\overline{X_\ba} = V_+(x_1^{a_1} + x_2^{a_2} + x_3^{a_3} + x_4^L) \subset \PPP(w_1,w_2,w_3,1)$ is quasismooth and well-formed; furthermore $X_\ba = X_{a_1,a_2,a_3} \isom D_+(x_4) \subset \overline{X_\ba}$. 
The complement $\overline{X_\ba} \setminus X_{\ba}$ is the curve $\overline{\partial_\ba} = V_+(x_1^{a_1} + x_2^{a_2} + x_3^{a_3} + x_4^L, x_4) \subseteq \PPP(w_1,w_2,w_3,1)$ which is isomorphic to the quasismooth (and hence smooth) curve $V_+(x_1^{a_1} + x_2^{a_2} + x_3^{a_3}) \subset \PPP(w_1,w_2,w_3)$. We will call $\overline{\partial_\ba}$ the \textit{boundary of $X_\ba$ in $\overline{X_\ba}$}. We let $\widetilde{X_\ba} \to \overline{X_\ba}$ denote the minimal resolution of singularities of $\overline{X_\ba}$ and note that $\widetilde{X_\ba}$ is a log smooth compactification of $X_\ba$. We let $\widetilde{\partial_\ba}$ denote the boundary divisor $\widetilde{X_\ba} \setminus X_\ba$ and we abbreviate its weighted graph $\Geul(X_\ba, \widetilde{X_\ba})$ by $\Geul(\ba)$ (as defined in \ref{weightedGraphDef}).       
\end{nothing}

\section{The Singularities of $\overline{X_\ba}$}\label{Section:Singularities}

\begin{nothing}\label{typeComputation}
    Consider $\overline{X_{\ba}} \subseteq \PPP(w_1, w_2,w_3,1)$ where $\overline{\partial_\ba} = V_+(x_4) \subset \overline{X_\ba}$. Note that the weight of $x_4$ is 1. Since $\overline{X_{\ba}}$ is quasismooth and well-formed, we have by Remark \ref{quasismoothRemark} that $\Sing \overline{X_{\ba}} = \Sing \PPP \cap \overline{X_{\ba}}$.  It can be checked that every singular point of $\overline{X_{\ba}}$ can be written in one of the following three forms:
    \begin{enumerate}[\rm(i)]
    \item $[1:y:0:0]$, $y \in \Comp^*$,
    \item $[1:0:z:0]$, $z \in \Comp^*$, 
    \item $[0:1:z:0]$, $z \in \Comp^*$. 
    \end{enumerate}
We want to compute the singularities of $\overline{X_\ba}$ and their types in two special cases:

    \begin{enumerate}[\rm(1)]
        \item $\ba = (a_1, a_2, a_3)$ where $\gcd(a_1 a_2,a_3) = 1$
        \item $\ba = (a_1,a_2, a_3) =  (ma_1',ma_2',ma_3')$ where $m \geq 1$ and $a_1',a_2',a_3'$ are pairwise relatively prime. 
    \end{enumerate}
    
We will compute the singular points and their types in Case (1) and will leave Case (2) to the reader. The arguments are the same.
\medskip

\noindent \textbf{Case (1).} Let $g_{i,j} = \gcd(a_i,a_j)$ where $1\leq i,j \leq 3$. We have $L = \frac{a_1 a_2 a_3}{g_{12}}$. We claim
    \begin{align}
        \label{g12} &\text{there are $g_{1,2}$ distinct points of type (i),}\\ 
        \label{g13} &\text{there is a unique point of type (ii),}\\
        \label{g23} &\text{there is a unique point of type (iii).}
    \end{align}
    We will prove \eqref{g12} and leave \eqref{g13} and \eqref{g23} to the reader. 
    Let $g = g_{1,2}$ and write $a_1 = ug, a_2 = vg$. Observe that every point of form (i) can be written as $[1: \zeta_{2a_2}^i: 0 :0]$ for some $i \in \{1,3, \dots, 2a_2-1\}$. We have $w_1 = L / a_1 = va_3$. The following are equivalent:
    
    \begin{itemize}
        \item $[1: (\zeta_{2a_2})^i: 0 :0] = [1:(\zeta_{2a_2})^{j}:0:0]$
        \item there exists some $\lambda \in \Comp^*$ such that $\lambda^{va_3} = 1$ and $\lambda^{ua_3} (\zeta_{2a_2})^i = (\zeta_{2a_2})^j$ 
        \item there exists $n$ such that $(\zeta_{va_3})^{nua_3} (\zeta_{2a_2})^{i-j} = 1$
        \item there exists $n$ such that $(\zeta_{v})^{nu} (\zeta_{2a_2})^{i-j} = 1$
        \item there exists $n$ such that $(\zeta_{2a_2v})^{2a_2un} (\zeta_{2a_2v})^{(i-j)v} = 1$
        \item there exists $n$ such that $(\zeta_{2a_2v})^{2a_2un+(i-j)v} = 1$
        \item there exists $n$ such that $a_2bv \mid 2a_2un + (i-j)v$
        \item there exists $n$ such that $a_2v \mid 2a_2un + \frac{(i-j)}{2}v$
        \item there exists $n$ such that $a_2v \mid 2vgun + \frac{(i-j)}{2}v$
        \item there exists $n$ such that $vg \mid 2gun + \frac{(i-j)}{2}$
        \item $g \mid \frac{i-j}{2}$ (because $\gcd(u,v) = 1$). 
    \end{itemize}
    This proves \eqref{g12}. Claims \eqref{g13} and \eqref{g23} follow by the same argument, where $\gcd(a_1, a_3) = 1$ is used in \eqref{g13} and $\gcd(a_2, a_3) = 1$ is used in \eqref{g23}.

    We will show that 
    \begin{align}
        \label{typeComputation1} \text{every point of form $[1:y:0:0]$ in $\overline{X_\ba}$ has type $\frac{1}{\gcd(w_1,w_2)}(w_3,1)$}, \\
        \label{typeComputation2} \text{every point of form $[1:0:z:0]$ in $\overline{X_\ba}$ has type $\frac{1}{\gcd(w_1,w_3)}(w_2,1)$}, \\
        \label{typeComputation3}
        \text{every point of form $[0:1:z:0]$ in $\overline{X_\ba}$ has type $\frac{1}{\gcd(w_2,w_3)}(w_1,1)$}.
    \end{align}

    It suffices to prove \eqref{typeComputation1}, the other two claims follow by symmetric arguments. 
    First, observe that the map $[x_1:x_2:x_3:x_4] \mapsto [\zeta_{a}x_1: x_2:x_3:x_4]$ is an automorphism of $\overline{X_\ba}$ that acts transitively on the points in $\overline{X_{\ba}}$ of form $[1:x_2:0:0]$, so every such point has the same type. Thus, it suffices to prove \eqref{typeComputation1} for the point $p = [1:\zeta_{2b}:0:0]$. The open set $D_+(x_1) \subset \overline{X_{\ba}} \subset \PPP(w_1,w_2,w_3,1)$ is the quotient $D_+(t) / \Gamma = V(1+x_2^{a_2}+x_3^{a_3} + x_4^L) / \Gamma$ where $D_+(t) \subset V_+(t^L + x_2^{a_2} + x_3^{a_3} + x_4^L) \subset \PPP(1,w_2,w_3,1)$ and $\Gamma = \Integ / w_1\Integ$ acts on $(x_2,x_3,x_4)\in V(1+x_2^{a_2}+ x_3^{a_3} + x_4^L)$ by $\epsilon \cdot (x_2,x_3,x_4) = (\epsilon^{w_2} x_2, \epsilon^{w_3} x_3, \epsilon x_4)$ where $\epsilon$ is a generator of $\Gamma$. The singular points of $D_+(x_1) \subset \overline{X_{\ba}}$ correspond to the orbits of the $\Gamma$-action on $D_+(t)$ with non-trivial stabilizer. In particular, the point $p =[1:\zeta_{2b}:0:0] \in D_+(x) = V(1+x_2^{a_2}+x_3^{a_3} + x_4^L) / \Gamma$ corresponds to the orbit of the action of $\Gamma$ that contains $\tilde{p} = (\zeta_{2b},0,0) \in V(1+x_2^{b}+x_3^{c} + x_4^L)$. The singular point of $D_+(x_1)$ corresponding to this $\Gamma$-orbit is analytically isomorphic to the quotient of the tangent space of $ V(1+x_2^{a_2}+x_3^{a_3} + x_4^L) \subset \PPP(1,w_2,w_3,1)$ at $(\zeta_{2b},0,0)$ by the induced action of the stabilizer of $(\zeta_{2b},0,0)$. 

    By \cite[Theorem 2]{prill1967local}  together with \cite[Lemme 1]{cartan1954}, the type of the singularity at $p \in D_+(x_1)$ is determined by the induced action of $\Gamma$ on the tangent space (or the cotangent space). The cotangent space $\mgoth_{\tilde{p}} / \mgoth^2_{\tilde{p}}$ is generated by the classes of $\bar{x}_3, \bar{x}_4$. The stabilizer of $p = (\zeta_{2b},0,0)$ under the $\Gamma$-action is the (cyclic) subgroup $\setspec{\epsilon^k}{\epsilon^{kw_2} = 1}$. This is a subgroup of $\Integ / w_1 \Integ$ of order $\gcd(w_1,w_2)$. Let $\mu = \epsilon^{w_1 / \gcd(w_1,w_2)}$ be a generator of this subgroup. Then $\mu$ acts on $\mgoth_{\tilde{p}} / \mgoth^2_{\tilde{p}} = (\bar{x}_3, \bar{x}_4)$ by $\mu \cdot  (\bar{x}_3, \bar{x}_4) = (\mu^{w_3} \bar{x}_3, \mu \bar{x}_4)$. This gives that action on the cotangent space $\mgoth_{\tilde p} / \mgoth^2_{\tilde p}$ is exactly the same as the action on the cotangent space at the origin of $\Aff^2$ that determines a $\frac{1}{\gcd(w_1,w_2)}(w_3,1)$ singularity, which proves \eqref{typeComputation1}. 

    We also remark that 
    \begin{equation}\label{yaxis}
        \text{the curve $V_+(x_4)$ in $\overline{X_{\ba}}$ corresponds to the $y$-axis of the $\frac{1}{\gcd(w_1,w_2)}(w_3,1)$ singularity.} 
    \end{equation}
\end{nothing}

This discussion proves Lemma \ref{otherTypes} below. The analysis of Case (2) yields Lemma \ref{mambmc}. These two lemmas are used throughout Section \ref{proofOfTheorem}. 

\begin{lemma}\label{otherTypes}
    Suppose $\ba = (a_1,a_2,a_3)$ is such that $a_1,a_2,a_3 \geq 2$ and $\gcd(a_1a_2,a_3) = 1$. Let $g = \gcd(a_1,a_2)$ and write $a_1 = ug$, $a_2 = vg$ where $\gcd(u,v) = 1$. Then
    \begin{enumerate}[\rm(a)]
    \item $\overline{X_\ba}$ has $g$ singular points of type $\frac{1}{a_3}(uvg,1)$.
    \item If $v \neq 1$, $\overline{X_\ba}$ has 1 singular point of type $\frac{1}{v}(\frac{L}{a_2},1)$. 
    \item If $u \neq 1$, $\overline{X_\ba}$ has 1 singular point of type $\frac{1}{u}(\frac{L}{a_1},1)$. 
    \end{enumerate}
\end{lemma}

\begin{lemma}\label{mambmc}
    Consider the tuple $(ma_1,ma_2,ma_3)$ where $a_1,a_2,a_3 \geq 1$ are pairwise relatively prime and $m \geq 1$. Let $n$ denote how many among $a_1,a_2,a_3$ equal 1. Then 
    \begin{enumerate}[\rm(a)]
        \item $\overline{X_{ma_1,ma_2,ma_3}} \subseteq \PPP(a_2a_3,a_1a_3,a_1a_2,1)$ has exactly $(3-n)m$ singular points.  
        \item If $a_3 \neq 1$, there are $m$ singular points of type $\frac{1}{a_3}(a_1a_2,1)$.
        \item If $a_2 \neq 1$, there are $m$ singular points of type $\frac{1}{a_2}(a_1a_3,1)$.
        \item If $a_1 \neq 1$, there are $m$ singular points of type  $\frac{1}{a_1}(a_2a_3,1)$. 
    \end{enumerate}
\end{lemma}

\section{Parameter-Closed Subsets}\label{paramClosed}
\begin{definition}
    Let $k \in \Nat$, and let $S \subseteq T \subseteq  (\Nat_{\geq 2})^k$. Let $\textbf{a} = (a_1, \dots, a_k), \textbf{b} = (b_1, \dots,b_k) \in T$. We say that \textit{$S$ is parameter-closed in $T$} if the following condition holds:
    \begin{equation}
        \text{if $\ba \in S$ and $X_\bb \isom X_\ba$ then $\bb \in S$.}
    \end{equation}
    In the special case where $T = (\Nat_{\geq 2})^k$, we will simply say \textit{$S$ is parameter-closed}.
\end{definition}

\begin{nothing}\label{parameterClosedRemarks}
    Fix $k \in \Nat$ and consider the following family of subsets of $(\Nat_{\geq 2})^k$ 
    \begin{equation} 
        \Cscr = \{\text{parameter-closed subsets of $(\Nat_{\geq 2})^k$}\}.
    \end{equation} 
    Observe that $\emptyset, (\Nat_{\geq 2})^k \in \Cscr$. We show that $\Cscr$ is closed both under arbitrary unions, and arbitrary intersections. Indeed, let $\{C_i\}_{i \in I}$ be a collection of parameter-closed sets. Suppose $\ba \in \cup_{i \in I} C_i$ and $X_\bb \isom X_\ba$. Then $\ba \in C_j$ for some $j \in I$ and since $C_j$ is parameter closed, $\bb \in C_j \subseteq \cup_{i \in I} C_i$. So $\cup_{i \in I} C_i$ is parameter-closed and hence that $\Cscr$ is closed under arbitrary unions. Similarly, suppose $X_\bb \isom X_\ba$ where $\ba \in \cap_{i \in I} C_i$. Since each $C_i$ is parameter-closed, $\bb \in C_i$ for all $i \in I$, so $\bb \in \cap_{i \in I} C_i$. This shows that $\Cscr$ is closed under arbitrary intersections. Thus, $\Cscr$ defines an Alexandrov topology on $(\Nat_{\geq 2})^k$. Note also that if $C \subseteq D$ and $C$ is parameter-closed in $D$, then $D \setminus C$ is parameter-closed in $D$. It is also easy to see that if $S$ is parameter-closed in $T$ and $T$ is parameter-closed in $U$, then $S$ is parameter-closed in $U$.  
\end{nothing}

\begin{example}\label{paramClosedExamples}
    Recall that a $\bk$-variety $X$ is \textit{rigid} if the only $\G_a$-action on $X$ is the trivial one. Define  
    $$S_{\text{\rm Dan}} =  \setspec{(a_1,a_2,a_3) \in (\Nat_{\geq 2})^3}{\text{$X_{a_1,a_2,a_3}$ is non-rigid}}.$$
    It is easy to see that $S_{\text{\rm Dan}}$
    is parameter-closed in $(\Nat_{\geq 2})^3$. Furthermore, we claim 
    $$
    S_{\text{\rm Dan}} = \setspec{(a_1,a_2,a_3) \in (\Nat_{\geq 2})^3}{a_i = a_j = 2 \text{ for some $i \neq j$}}.$$ 
    For $(\supseteq)$, it is easy to check that $X_{2,2,a_3}$ is isomorphic to a Danielewski surface and hence is not rigid. The inclusion $(\subseteq)$ follows by combining \cite[Lemma 4]{Kali-Zaid_2000}, with \cite[Proposition 10.5]{daigle2023rigidity}. We also note that  
    \begin{align*}
        S'_{\text{\rm Dan}} &= \setspec{(a_1,a_2,a_3,a_4)\in (\Nat_{\geq 2})^4}{\text{$X_{a_1,a_2,a_3,a_4}$ is non-rigid}} \\
        &=  \setspec{(a_1,a_2,a_3,a_4) \in (\Nat_{\geq 2})^4}{a_i = a_j = 2 \text{ for some $i \neq j$}}.
    \end{align*}
    is parameter-closed in $(\Nat_{\geq 2})^4$ where the second equality follows by the same argument, appealing to \cite[Main Theorem]{chitayat2025rigid} instead of \cite[Lemma 4]{Kali-Zaid_2000}.  
\end{example}

\section{Proof of the Main Theorem}\label{proofOfTheorem}

\begin{notation}
    Let $S \subseteq (\Nat_{\geq 2})^3$. We say that \textit{``$S$ is good"} if $\ba, \bb \in S$ and $X_\ba \isom X_\bb$ imply that $\ba \sim \bb$.
\end{notation}

\begin{lemma}\label{reduction}
    Let $S = \sqcup_{i \in I} S_i$ be a partition of $S$. Suppose that 
    \begin{enumerate}[\rm(i)]
        \item each $S_i$ is parameter-closed in $S$
        \item $S_i$ is good for all $i \in I$. 
    \end{enumerate}
    Then $S$ is good.  
\end{lemma}
\begin{proof}
    Let $\ba, \bb \in S$ and assume $X_\bb \isom X_\ba$. Since the sets $S_i$ partition $S$, $\ba \in S_i$ for a unique $S_i$. Since the $S_i$ are parameter-closed, $\bb \in S_i$ and since $S_i$ is good $\bb \sim \ba$.
\end{proof}

\begin{lemma}\label{boundaryGenus}
    Let $\ba, \bb \in (\Nat_{\geq 2})^3$. If $X_\ba \isom X_\bb$, then the boundary curves $\overline{\partial_\ba} \subset \overline{X_\ba}$ and $\overline{\partial_\bb} \subset \overline{X_\bb}$ have the same geometric genus. 
\end{lemma}
\begin{proof}
    The log smooth compactifications  $\widetilde{X_\ba}$ and $\widetilde{X_\bb}$ are obtained by resolving cyclic quotient singularities along the smooth curves $\overline{\partial_\ba}$ and $\overline{\partial_\bb}$. In particular the boundaries $\widetilde{\partial_\ba}$ and $\widetilde{\partial_\bb}$ have at most one curve of non-zero genus (the strict transform of $\overline{\partial_\ba}$) in their support. The result follows from \ref{weightedGraphDef} (b). 
\end{proof}

\begin{lemma}\label{rational}
    Let $\ba = (a_1,a_2,a_3) \in (\Nat_{\geq 2})^3$. 
    \begin{enumerate}[\rm(a)]
        \item The following are equivalent:
    
    \begin{enumerate}[\rm(i)]
        \item $X_\ba$ is rational,
        \item $\frac{1}{a_1} + \frac{1}{a_2} + \frac{1}{a_3} + \frac{1}{L} > 1$,
        \item the amplitude $\alpha_\ba < 0$,
        \item $\ba \in \{(2,2,a_3), (2,3,3), (2,3,4), (2,3,5), (2,3,6),(2,4,4), (3,3,3)\}$.
    \end{enumerate}
    \item The set $S_{\text{\rm rat}} = \setspec{\ba}{\alpha_\ba < 0}$ is parameter-closed. 
    \item The subset $S_{\text{\rm Dan}}$ is good. 
    \item The subset $S_{\text{\rm rat}}$ is good.
    
    \end{enumerate}
    
\end{lemma}

\begin{proof}
    We have $X_\ba$ is rational if and only if  $\overline{X_\ba} = V_+(x_1^{a_1} + x_2^{a_2} + x_3^{a_3} + x_4^L) \subset \PPP(w_1, w_2, w_3, 1)$ is rational if and only if $\frac{1}{a_1} + \frac{1}{a_2} + \frac{1}{a_3} + \frac{1}{L} > 1$ where $(\Leftarrow)$ is by \cite[Corollary 4.9]{chitayat2025rationality}  and $(\Rightarrow)$ is because if the amplitude $\alpha_\ba \geq 0$, then by Remark \ref{quasismoothRemark} (e),  we have $p_g(X) = \dim_\Comp B_\alpha \geq 1$, because $w_4 = 1$. This proves ${\rm(i)} \iff {\rm (ii)}$. The equivalences ${\rm (ii)} \iff {\rm (iii)}$ and ${\rm (ii)} \iff {\rm (iv)}$ are elementary number theory and are left to the reader. This proves (a). For (b), the equality of sets is by (a) and parameter-closedness is obvious. 

    We prove (c). Let $\ba \in S_{\text{\rm Dan}}$, and assume $X_\bb \isom X_\ba$. Then $\bb \in S_{\text{\rm Dan}}$ by Example \ref{paramClosedExamples}, so we may assume $a_1 =a_2 = b_1 =b_2 = 2$. It suffices to prove that if $X_{2,2,a_3} \isom X_{2,2,b_3}$, then $a_3 = b_3$. By Table \ref{invariants} in \ref{onlyTau}, $\rank(\Cl(X_{2,2,c})) = c-1$ for all $c \in \Nat_{\geq 2}$, so $a_3 = b_3$ and $\ba \sim \bb$. 
    For (d), suppose $\ba \in S_{\text{\rm rat}}$ and assume $X_\bb \isom X_\ba$. Since $S_{\text{\rm rat}} \setminus S_{\text{\rm Dan}}$ is parameter-closed by (b) together with Example \ref{paramClosedExamples}, $\bb \in S_{\text{\rm rat}} \setminus S_{\text{\rm Dan}}$. By the invariants computed in Table \ref{invariants} in \ref{onlyTau}, no distinct elements of $S_{\text{\rm rat}} \setminus S_{\text{\rm Dan}}$ are isomorphic, proving (d). 
\end{proof}

\begin{nothing}\label{onlyTau}
    Let $\ba \in (\Nat_{\geq 2})^3$ and recall that $C_{\overline{X_\ba}} = \Spec(\Comp[x_1,x_2,x_3,x_4] / \lb x^{a_1} + x_2^{a_2} + x_3^{a_3} + x_4^{L} \rb)$ denotes the affine cone over $\overline{X_\ba}$. Since there is no 4-tuple $(m_1, m_2, m_3, m_4) \in (\Nat_{\geq 1})^4$ that satisfies $\sum_{i = 1}^4 \frac{m_i}{a_i}=1$, \cite[Satz 2.1]{Storch} implies that $\Cl(C_{\overline{X_\ba}}) = \Integ^{k}$ and $\Cl(\overline{X_\ba}) \isom \Integ^{k+1}$ where $k = \tau(a_1, a_2, a_3, L)$ and $\tau : (\Nat_{\geq 1})^4 \to \Nat$ is defined at the bottom of \cite[p.191]{Storch}. We have that $V_+(x_4)$ is an irreducible closed subset of $\overline{X_\ba}$ of codimension 1, so  by \cite[Proposition 6.5]{Hartshorne},  $\Cl(X_\ba) \isom \Integ^{k+1} / \lb V_+(x_4) \rb$ where $\lb V_+(x_4) \rb$ is the subgroup of $\Cl(\overline{X_\ba}) \isom \Integ^{k+1}$ generated by $V_+(x_4)$. In particular, the rank of $\Cl(X_\ba)$ as an abelian group is $k = \tau(a_1, a_2, a_3, L)$. We obtain the following table of invariants of $X_\ba$ for each $\ba \in S_{\text{\rm rat}}$. 

\begin{table}[h]
\centering

\begin{tabular}{|c|c|c|c|c|}
\hline
$\ba = (a_1,a_2,a_3)$ & $\Cl(C_{\overline{X_\ba}})$ & $\Cl(\overline{X_\ba})$  
& $\rank \Cl(X_{\ba})$ 
& $p_g(\overline{\partial_\ba})$ \\
\hline
$(2,2,a_3)$ & $\Integ^{a_3-1}$ & $\Integ^{a_3}$ & $a_3-1$ & $0$ \\ \hline
$(2,3,3)$ & $\Integ^4$ & $\Integ^5$ & $4$ & $0$ \\ \hline
$(2,3,4)$ & $\Integ^6$ & $\Integ^7$ & $6$ & $0$ \\ \hline
$(2,3,5)$ & $\Integ^8$ & $\Integ^9$ & $8$ & $0$ \\ \hline
$(2,3,6)$ & $\Integ^8$ & $\Integ^9$ & $8$ & $1$\\ \hline
$(2,4,4)$ & $\Integ^7$ & $\Integ^8$ & $7$ & $1$ \\ \hline
$(3,3,3)$ & $\Integ^6$ & $\Integ^7$ & $6$ & $1$ \\ \hline 
\end{tabular}
\caption{Invariants of $X_{a_1,a_2,a_3}$ where $(a_1,a_2,a_3) \in S_{\text{\rm rat}}$}
\label{invariants}
\end{table}
\vspace{-20pt}

\end{nothing}
\begin{lemma}\label{K3}
     Let $\ba = (a_1, a_2,a_3) \in (\Nat_{\geq 2})^3$,  $L = \lcm(a_1,a_2,a_3)$ and let $w_i = L / a_i$ for each $i \in \{1,2,3\}$.
    The following are equivalent:
    \begin{enumerate}[\rm(a)]
        \item $L - w_1  - w_2 - w_3 - 1 = 0$,
        \item $\frac{1}{a_1} + \frac{1}{a_2} + \frac{1}{a_3} + \frac{1}{L} = 1$,
        \item up to a permutation of entries
        $$(a_1,a_2,a_3) \in \{(2,3,7), (2,3,8), (2,3,9), (2,3,12), (2,4,5), (2,4,6), (2,4,8), (2,5,5), (2,6,6),(3,3,6), (4,4,4)\}.$$
    \end{enumerate} 
    Furthermore, the rank of the divisor class group of $X_{\ba}$ is given in Table \ref{K3invariants} below. Furthermore, no two distinct surfaces $X_{\ba}$ from the list that satisfy $p_g(\overline{\partial_\ba}) = 0$ are isomorphic.  
\end{lemma}

\begin{table}[h]
\centering
\begin{tabular}{|c|c|c|c|c|c|}
\hline
$\ba = (a_1,a_2,a_3)$ & $\phi(L)$ & $\tau(C_{\overline{X_\ba}})$ & $\Cl(\overline{X_\ba})$ & $\rk(\Cl(X_\ba))$ & $p_g(\overline{\partial_\ba})$ \\
\hline
$(2,3,7)$ & 12 & 12 & $\Integ$ & 0 & 0\\ \hline
$(2,3,8)$ & 8 & 14 & $\Integ^7 $ & 6 & 0 \\ \hline
$(2,3,9)$ & 6 & 16 & $\Integ^{11}$& 10 & 0 \\ \hline
$(2,3,12)$& 4 & 20 & $\Integ^{17}$ & 16  & 1\\ \hline
$(2,4,5)$ & 8 & 12 & $\Integ^5$& 4 & 0\\ \hline
$(2,4,6)$ & 4 & 15 & $\Integ^{12}$ & 11 & 0 \\ \hline
$(2,4,8)$ & 4 & 19 & $\Integ^{16}$ & 15 & 1 \\ \hline
$(2,5,5)$ & 4 & 16 & $\Integ^{13}$ & 12 & 0 \\ \hline
$(2,6,6)$ & 2 &  21 & $\Integ^{20}$ & 19 & 2\\ \hline
$(3,3,6)$ & 2 &  17 & $\Integ^{16}$ & 15 & 1\\ \hline
$(4,4,4)$ & 2 &  21 & $\Integ^{20}$ & 19 & 3\\ \hline
\end{tabular}
\caption{Invariants of $X_{a_1,a_2,a_3}$ where $\alpha =L - w_1 - w_2 - w_3 - 1 = 0$}
\label{K3invariants}
\end{table}
\begin{proof}
    We leave the equivalences (a) $\iff$ (b) $\iff$ (c) to the reader. Let $\phi$ denote Euler's totient function. The symbol $\tau$ is defined in \cite[p.191]{Storch}. Using \cite[Satz 2.1]{Storch}, we find (in these special cases) that $\Cl(C_{\overline{X_\ba}}) = \Integ^k$ where $k = \tau(C_{\overline{X_\ba}}) - \phi(L)$.  and since $\Cl(\overline{X_\ba}) \isom \Cl(C_{\overline{X_\ba}}) + \Integ$ we obtain the value of $\Cl(\overline{X_\ba})$ in Table \ref{K3invariants}. Since $X_\ba = D_+(x_4) = \overline{X_\ba} \setminus \overline{\partial_\ba}$ where $\overline{\partial_\ba} = V_+(x_4)$ is irreducible, \cite[Proposition 6.5]{Hartshorne} gives $\rk(\Cl(X_\ba)) = \rk(\Cl(\overline{X_\ba})) - 1$. It now suffices to examine the last two columns of Table \ref{K3invariants}.   
\end{proof}

\begin{nothing}\label{specialSets}
    Define the sets 
    \begin{itemize}
        \item $S_0 = \setspec{\ba \in (\Nat_{\geq 2})^3}{p_g(\overline{\partial_\ba}) = 0}$,
        \item $S_+ = \setspec{\ba \in (\Nat_{\geq 2})^3}{p_g(\overline{\partial_\ba}) > 0}$,
        \item $S_{+,+}= \setspec{\ba \in S_+}{L - w_1 - w_2 - w_3 - 1 > 0}$,
        \item $S_{+,0} = \setspec{\ba \in S_+}{L - w_1 - w_2 - w_3 - 1 = 0}$,
        \item $S_{+,-} = \setspec{\ba \in S_+}{L - w_1 - w_2 - w_3 - 1 < 0} = S_+ \cap S_{\text{\rm rat}}$,
        
    \end{itemize}
\end{nothing}

\begin{lemma}\label{theEasyPart}
    Consider the sets $S_0, S_{+,+}, S_{+,0}, S_{+,-}$ from \ref{specialSets}.
    \begin{enumerate}[\rm(a)]
        \item The sets $S_0, S_{+,+}, S_{+,0}, S_{+,-}$ partition $(\Nat_{\geq 2})^3$ and are parameter-closed. 
        \item The subset $S_{+,+}$ is good.
        \item The subset $S_{+,0}$ is good.
        \item The subset $S_{+,-}$ is good.
    \end{enumerate}
\end{lemma}
\begin{proof}
    We prove (a). First, we have $(\Nat_{\geq 2})^3 = S_0 \sqcup S_+ = S_0 \sqcup S_{+,+} \sqcup S_{+,0} \sqcup S_{+,-}$, proving the first claim. 
    
    We prove that $S_0$ is parameter closed. Let $\ba \in S_0$ and suppose $\varphi : X_\bb \to X_\ba$ is an isomorphism. We must show that $p_g(\overline{\partial_\bb}) = 0$. If $p_g(\overline{\partial_\bb}) > 0$, then Proposition \ref{mapExtends} implies that the birational map $\varphi :\overline{X_\bb} \dashrightarrow \overline{X_\ba}$ restricts to a birational map $\varphi|_{\overline{\partial_\bb}} : \overline{\partial_\bb} \dashrightarrow \overline{\partial_\ba}$. In particular, this restriction map extends to birational morphism between $\overline{\partial_\bb}$ and $\overline{\partial_\ba}$ and so $\overline{\partial_\bb} \isom \overline{\partial_\ba}$ and hence $0 < p_g(\overline{\partial_\bb}) = p_g(\overline{\partial_\ba}) = 0$. This is absurd, so $p_g(\overline{\partial_\bb}) = 0$. This proves that $S_0$ is parameter-closed and it follows from \ref{parameterClosedRemarks} that $S_+$ is parameter-closed. Since $S_+$ is parameter-closed and $S_{\text{\rm rat}}$ is parameter-closed by Lemma \ref{rational} (b), $S_{+,-} = S_+ \cap S_{\text{\rm rat}}$ is parameter-closed. To prove (a), it now suffices to prove 
    \begin{equation}\label{sufficient}
        \text{$S_{+,+}$ is parameter-closed in $S_+$.}
    \end{equation}
    Indeed if \eqref{sufficient} is true, then $S_{+,+}$ is parameter-closed by \ref{parameterClosedRemarks} and hence $S_0,S_{+,+}, S_{+,-}$ are all parameter-closed and so the complement of their union (which is $S_{+,0}$) is as well. 
    We prove \eqref{sufficient}. Let $\ba \in S_{+,+}$, and suppose $\bb \in S_+$ is such that $\varphi: X_\bb \overset{\isom}{\to} X_\ba$. Then $p_g(C_\bb) > 0$ and so by Proposition \ref{mapExtends} the induced rational map  $\varphi: \overline{X_\ba} \dashrightarrow \overline{X_\bb}$ is an isomorphism away from at most finitely many points of $\overline{X_\bb}$ and $\overline{X_\ba}$. Since the canonical divisor of $\overline{X_\ba}$ is ample (by Remark \ref{quasismoothRemark} (d)),  Proposition \ref{IsoGeneralization} (a) implies that $K_{\overline{X_\bb}}$ is non-trivial. Consequently $0 \neq \alpha_\bb > 0$ so $\bb \in S_{+,+}$. This proves \eqref{sufficient} and consequently proves (a).

    We prove (b). Write $\bb = (b_1,b_2,b_3)$ and let $L' = \lcm(b_1, b_2,b_3)$. Assume $\ba, \bb \in S_{+,+}$ and that $X_\ba \isom X_\bb$. By Proposition \ref{IsoGeneralization} (b) $\overline{X_\ba}$ and $\overline{X_\bb}$ are isomorphic. By Theorem \ref{esser}, $L = L'$ and, up to permutation, $L/a_i = L'/b_i = L / b_i$ for each $i = 1,2,3$. Consequently, $\ba \sim \bb$ proving (b).
    
    To prove (c), observe that the elements of $S_{+,0}$ all appear in Table \ref{K3invariants}. Using Lemma \ref{boundaryGenus}, the values of $p_g(C_\ba)$ and $\Cl(X_\ba)$ given in Table \ref{K3invariants} show that if $\ba,\bb \in S_{+,0}$ are such that $X_\ba \isom X_\bb$ then $\ba \sim \bb$. This proves (c). Since $S_{+,-} \subseteq S_{\text{\rm rat}}$, (d) follows immediately from Lemma \ref{rational}. 
\end{proof}

\begin{remark}\label{easyDeduction} In view of Lemmas \ref{reduction} and \ref{theEasyPart}, to finish the proof of the Main Theorem, it suffices to prove that $S_0$ is good.
\end{remark} 

We characterize the tuples $\ba = (a_1,a_2,a_3)$ in $S_0$.

\begin{proposition}\label{S0char}
    The following are equivalent:
    \begin{enumerate}[\rm(a)]
        \item $\ba \in S_0$
        \item The boundary curve of $X_\ba$ in $\overline{X_\ba}$ given by $V_+(x_4) \subset \overline{X_\ba}$ is a smooth rational curve.
        \item The surface $\Spec(\Comp[x_1,x_2,x_3] / \lb x_1^{a_1} + x_2^{a_2} + x_3^{a_3} \rb)$ is rational.
        \item Up to a permutation of $(a_1,a_2,a_3)$, one of the following holds:
        \begin{enumerate}[\rm(i)]
            \item $\gcd(a_1a_2, a_3) = 1$
            \item $(a_1,a_2,a_3) = (2a_1',2a_2',2a_3')$ where $a_1',a_2',a_3'$ are pairwise relatively prime.  
        \end{enumerate}
        
    \end{enumerate}
\end{proposition}
\begin{proof}
    The equivalence (a) $\iff$(b)  is the definition of $S_0$; (b) implies (c) because the cone over a rational curve is rational and the converse follows because stable rationality implies rationality in dimension 1. The equivalence (c) $\iff$ (d) is the special case of \cite[Proposition 5.5]{arzhantsev2018log},  applied to Pham-Brieskorn surfaces. Note that our condition (i) is equivalent to their condition (i).     
\end{proof}

Recall the definition of $\Geul(\ba)$ from Section \ref{section:Setup}. 

\begin{lemma}\label{3branches}
    Let $\ba \in (\Nat_{\geq 2})^3$ and let $r \geq 3$. 
    \begin{enumerate}[\rm(a)]
        \item The following are equivalent:
        \begin{enumerate}[\rm(i)]
        \item The weighted graph  $\Geul(\ba)$ is star-shaped with $r$ branches. 

        \item $\overline{X_\ba}$ has $r$ singular points
    \end{enumerate}

    \item Assume in addition that conditions {\rm(i)} and {\rm(ii)} are satisfied, then
        \begin{enumerate}[\rm(i)]
        \item $\Geul(\ba)$ is the unique minimal graph in its equivalence class. 
        \item If $X_\ba \isom X_\bb$ then $\Geul(\ba) = \Geul(\bb)$. In particular, $\Geul(\ba)$ and $\Geul(\bb)$ have the same number of branches. 
        \item Each branch of the star-shaped graph
     
    
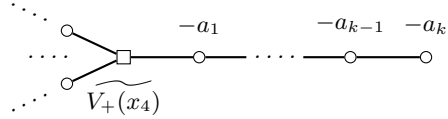
\begin{center}
\begin{tikzpicture}[
    every node/.style={font=\small},
    openr/.style={rectangle,draw,inner sep=2.5pt},
    filledr/.style={rectangle,fill,inner sep=2.5pt},
    openc/.style={circle,draw,inner sep=1.5pt},
    filledc/.style={circle,fill,inner sep=1.5pt},
]

\node[openc] (b1) at (-0.75,0.35) {};
\node[rotate=-25] at (-1.25,0.6) {.\,.\,.\,.};
\node[openc] (c1) at (-0.75,-0.35) {};
\node[rotate=25] at (-1.25,-0.6) {.\,.\,.\,.};

\node[openr] (a0) at (0,0) {};
\node[openc] (ar) at (1,0) {};
\node[] (adummy1) at (1.75,0) {};

\node at (2,0) {.\,.\,.\,.};
\node at (-1,0) {.\,.\,.\,.};

\node[] (adummy2) at (2.25,0) {};
\node[openc] (a2) at (3,0) {};
\node[openc] (a1)  at (4,0) {};

\draw[thick] (a0)--(ar)--(adummy1);
\draw[thick] (adummy2)--(a2)--(a1);
\draw[thick] (b1)--(a0);
\draw[thick] (c1)--(a0);

\node[above=5pt] at (ar) {$-a_1$};
\node[above=5pt] at (a2) {$-a_{k-1}$};
\node[above=5pt] at (a1) {$-a_k$};
\node[below=5pt] at (a0) {$\widetilde{V_+(x_4)}$};
\end{tikzpicture}

\end{center}
     of $\Geul(\ba)$ corresponds to a $\frac{1}{n}(m,1)$ singularity lying on $\overline{\partial_\ba} = V_+(x_4) \subset \overline{X_\ba}$, where $[a_1,a_2,\dots, a_r] = \frac{n}{m}$ and $\gcd(n,m) = 1$.
        \end{enumerate}
    \end{enumerate}
\end{lemma}

\begin{proof}
     The singular points of $\overline{X_\ba}$ are all cyclic quotient singularities, and they lie along the smooth curve $\overline{\partial_\ba} = V_+(x_4)$. For each such singular point, by \eqref{yaxis} in \ref{typeComputation},  $V_+(x_4)$ corresponds to the $y$-axis of the $\frac{1}{n}(m,1)$ singularity in question. Using the notation of \ref{rationalSingularities}, it follows that $\widetilde{V_+(x_4)}$ intersects the curve $E_1$ transversely. Using that $n \geq 3$ together with \ref{rationalSingularities}, part (a) follows. 
    
    We prove (b). Part (i) follows immediately from Lemma \ref{minimalGraph3branches}. For (ii), since $X_\ba \isom X_\bb$, $\Geul(\ba) \sim \Geul(\bb)$. If $\Geul(\bb) \neq \Geul(\ba)$, then $\Geul(\bb)$ is not minimal in its equivalence class. Part (a) together with part (b-i) implies that $\overline{X_\bb}$ has at most 2 singular points. But then $\Geul(\bb)$ is a chain, contradicting the assumption that $\Geul(\ba) = \Geul(\bb)$. This proves part (ii). Part (iii) is also an immediate consequence of the discussion in \ref{rationalSingularities}.   
\end{proof}

\begin{nothing}
    By Proposition \ref{S0char} $S_0$ is the disjoint union of the following two subsets:
    \begin{itemize}
        \item $T_1 = \setspec{(a_1,a_2,a_3) \in S_0}{\gcd(a_1a_2,a_3) = 1}$,
        \item $T_2 = \setspec{(a_1,a_2,a_3) \in S_0}{(a_1,a_2,a_3) = (2a_1',2a_2',2a_3') \text{ and $a_1',a_2',a_3'$ are pairwise relatively prime}}$.
    \end{itemize}
\end{nothing}

\begin{lemma}\label{ConjecturevalidOnS0}
    The following hold:
    \begin{enumerate}[\rm(a)]
        \item The sets $T_1$ and $T_2$ are each parameter-closed in $S_0$.
        \item The subset $T_1$ is good.
        \item The subset $T_2$ is good. 
        \item The subset $S_0$ is good. 
    \end{enumerate}
\end{lemma}

\begin{proof}
    For (a), since $T_2 = S_0 \setminus T_1$, it suffices to prove that $T_1$ is parameter-closed in $S_0$. Let $\ba = (a_1,a_2,a_3) \in T_1$ where $g = \gcd(a_1,a_2)$, $a_1 = ug$, $a_2 = vg$ and assume $X_\bb \isom X_\ba$. Then $\Geul(\bb) \sim \Geul(\ba)$. Since $a_1,a_2,a_3 \geq 2$, exactly one of the following four conditions is satisfied: 
    \begin{enumerate}[\rm(i)]
        \item $g = 1$ and $u,v \geq 2$
        \item $g = 2$ and $u,v = 1$
        \item $g \geq 2$ and $u,v$ are not both 1
        \item $g \geq 3$ and $u = v = 1$
    \end{enumerate}
    Assume (i) holds. By Lemma \ref{otherTypes}, $\overline{X_\ba}$ has exactly three singular points and so  $\Geul(\ba) = \Geul(\bb)$ and $\overline{X_\bb}$ has three singular points by Lemma \ref{3branches}. By Lemma \ref{mambmc}, for all $\bb \in T_2$, $\overline{X_\bb}$ has an even number of singular points, so $\bb \notin T_2$, so $\bb \in T_1$.

    For case (ii), we prove the stronger claim:
    \begin{equation}\label{caseii}
        \text{if $\ba \in T_1$ satisfies (ii) and $X_\bb \isom X_\ba$, then $\bb\ \sim \ba$.}
    \end{equation}

    Indeed, assume $\ba$ satisfies (ii). Then $\ba = (2,2,a_3)$ where $a_3$ is odd and so $\ba \in S_{\text{\rm Dan}}$. Since $S_{\text{\rm Dan}}$ is parameter-closed by Example \ref{paramClosedExamples}, $\bb \in S_{\text{\rm Dan}}$. Since $S_{\text{\rm Dan}}$ is good, we have $\bb \sim \ba$. This proves \eqref{caseii}.   

    Assume (iii) holds, and suppose $X_\bb \isom X_\ba$. Since $X_\bb \isom X_\ba$ and $\overline{X_\ba}$ has at least 3 singular points, Lemma \ref{3branches} implies that $\overline{X_\ba}$ and $\overline{X_\bb}$ have $g$ singular points of type $\frac{1}{a_3}(uvg,1)$. Assume $\bb \in T_2$. By Lemma \ref{mambmc}, if $\overline{X_\bb}$ has a singular point of a specific type, it has exactly $2$ singular points of that type. Consequently, we must have $u = v = 1$, otherwise $\overline{X_\ba}$ and $\overline{X_\bb}$ would have only one singular point of type either $\frac{1}{v}(\frac{L}{a_2},1)$ or $\frac{1}{u}(\frac{L}{a_1},1)$ (by Lemma \ref{otherTypes}). Since $u = v = 1$ contradicts (iii),  we deduce $\bb \notin T_2$, proving (a) in case (iii). 
    
    Finally assume (iv) holds. By Lemma \ref{3branches}, $\overline{X_\ba}$ has at least three singular points of the same type. If $X_\bb \isom X_\ba$ then the same is true for $\overline{X_\bb}$. Lemma \ref{mambmc} then implies $\bb \notin T_2$. So $\bb \in T_1$, proving (a) in case (iv). This completes the proof of (a). 
     
    We prove (b). Let $\ba \in T_1$ and write $\ba = (a_1,a_2,a_3) = (ug,vg,a_3)$ where $g = \gcd(a_0,a_1)$. Assume $\bb = (b_1,b_2,b_3) \in T_1$, where $g' = \gcd(b_1,b_2)$, $b_1 = u'g'$, $b_2 = v'g'$ and $X_\bb \isom X_\ba$. 
    
    If $\ba$ is as in case (i) then by Lemma \ref{mambmc} $\overline{X_\ba}$ has 1 singular point of each type $\frac{1}{a_3}(a_1a_2,1), \frac{1}{a_2}(a_1a_3,1)$ and $\frac{1}{a_1}(a_2a_3,1)$. So $\Geul(\ba) = \Geul(\bb)$ and $\overline{X_\ba}$ and $\overline{X_\bb}$ have the same types of singular points. Since $a_1,a_2,a_3$ are distinct, it follows from \ref{rationalSingularities} that $\bb \sim \ba$, handling case (i). Statement \eqref{caseii} proves case (ii). Assume $\ba$ is as in case (iii). Then by Lemma \ref{otherTypes}, $\overline{X_\ba}$ has $g$ singular points of type $\frac{1}{a_3}(guv,1)$ and at least one other singular point. Since $\Geul(\ba) = \Geul(\bb)$ has at least 3 branches, Lemma \ref{3branches}(b) implies that $b_3 = a_3$ and that $gu'v' = guv$ and so $uv = u'v'$. Since $\gcd(u,v) = \gcd(u',v') = 1$, either 
    \begin{itemize}
        \item $u = u'$ and $v = v'$ and hence $a_1 = gu = gu' = b_1$ and $a_2 = gv = gv' = b_2$ or
        \item $u = v'$ and $v = u'$ and hence $a_1 = gu = gv' = b_2$ and $a_2 = gv = gu' = b_1$. 
    \end{itemize}
    In either case, we have $\ba \sim \bb$. This proves the result if $\ba$ is as in (iii). Finally, assume $\ba$ is as in (iv). Then $\ba = (g,g,a_3)$ where $\gcd(g,a_3) = 1$. Since $g \geq 3$, by Lemma \ref{otherTypes}, $\overline{X_\ba}$ has $g$ singular points of type $\frac{1}{a_3}(g,1)$ and no other singular points.  Assume $X_\bb \isom X_\ba$. Then by Lemma \ref{3branches} (b), $\Geul(\bb) = \Geul(\ba)$ and $\overline{X_\bb}$ also has $g$ singular points of type $\frac{1}{a_3}(g,1)$ and no others. It follows from Lemma \ref{otherTypes}, that $g' = g$ and that $u' = 1$ and $v' = 1$ and so $a_1 = a_2 = g = g' = b_1 = b_2$. Then \ref{rationalSingularities} gives $a_3 = b_3$ and so $\ba \sim \bb$. This proves (b).

    We prove (c). Suppose $\ba = (2a_1', 2a_2', 2a_3'), \bb = (2b_1',2b_2',2b_3') \in T_2$ and assume $X_\bb \isom X_\ba$. For $x \in \Nat$, consider the indicator function 
    \[
{\bf 1}_{x} =
\begin{cases}
  1 & \text{if } x = 1 \\
  0 & \text{otherwise}
\end{cases}
\]  
    Let $n_\ba = {\bf 1}_{a_1'} + {\bf 1}_{a_2'} + {\bf 1}_{a_3'}$ and let $n_\bb = {\bf 1}_{b_1'} + {\bf 1}_{b_2'} + {\bf 1}_{b_3'}$. By Lemma \ref{mambmc}, the number of branches of $\Geul(\ba)$ is $6 - 2n_\ba$ and the number of branches of $\Geul(\bb) = 6 - 2n_\bb$. By Lemma \ref{3branches}, the following hold:
    \begin{enumerate}
        \item $n_\ba \geq 2 \iff n_\bb \geq 2$,
        \item $n_\ba = 1 \iff n_\bb = 1$,
        \item $n_\ba = 0 \iff n_\bb = 0$.
    \end{enumerate}
    If $n_\ba \geq 2$, then $X_\ba = X_{2,2,2a_3'}$ and consequently $X_\ba \in S_{\text{\rm Dan}}$. Since $S_{\text{\rm Dan}}$ is parameter-closed, it is also parameter-closed in $T_2$. It follows that $\bb \in S_{\text{\rm Dan}}$ and since (by Lemma \ref{rational} (c)) $S_{\text{\rm Dan}}$ is good, we have $\bb \sim \ba$ and the proof is complete in case (1).

    If $n_\ba = 1$, then $n_\bb = 1$. We may assume without loss of generality that $a_1' = 1$, $b_1' = 1$ in which case Lemma \ref{mambmc} gives that $X_\ba$ has 2 singular points of type $\frac{1}{a_3'}(a_2',1)$ and 2 singular points of type $\frac{1}{a'_2}(a_3',1)$. Similarly, $X_\bb$ has 2 singular points of type $\frac{1}{b_3'}(b_2',1)$ and 2 singular points of type $\frac{1}{b'_2}(b_3',1)$. Since $\Geul(\bb) = \Geul(\ba)$ and  $\overline{X_\ba}$ and $\overline{X_\bb}$ have the same types of singular points, it follows from \ref{rationalSingularities} that either $b_3' = a_3'$ and $b_2' = a_2'$ or $b_3' = a_2'$ and $b_2' = a_3'$. In either case, we have $\ba \sim \bb$. This handles case (2). 

    If $n_\ba = 0$, then $n_\bb = 0$ and Lemma \ref{mambmc} gives that $X_\ba$ has 2 singular points of type $\frac{1}{a_3'}(a_1'a_2',1)$, 2 singular points of type $\frac{1}{a_2'}(a_1'a_3',1)$ and 2 singular points of type $\frac{1}{a_1'}(a_2'a_3',1)$. Similarly, $X_\bb$ has 2 singular points of type $\frac{1}{b_3'}(b_1'b_2',1)$, 2 singular points of type $\frac{1}{b_2'}(b_1'b_3',1)$ and 2 singular points of type $\frac{1}{b_1'}(b_2'b_3',1)$. The same argument as in the $n_\ba = n_\bb = 1$ case gives that $\ba \sim \bb$. This proves case (3), and consequently proves (c).   

    Part (d), follows from (a),(b),(c) together with Lemma \ref{reduction}.  
\end{proof}

\begin{theorem}\label{mainTheoremText}
    Let $\ba, \bb \in (\Nat_{\geq 2})^3$. If $X_\ba \isom X_\bb$, then $\ba \sim \bb$. 
\end{theorem}
\begin{proof}
    By Remark \ref{easyDeduction} and Lemma \ref{ConjecturevalidOnS0} (d).  
\end{proof}

\begin{remark}
    For each $m \geq 1$, the surface $D_m = \Spec (\Comp[x_1,x_2,x_3] / \lb x_1^m x_2 + x_3^{2} + 1 \rb$ is non-rigid. By Example \ref{paramClosedExamples}, $D_m$ is isomorphic to $X_{\ba}$ for some $\ba \in (\Nat_{\geq 2})^3$ if and only if $\ba \in S_{\text{\rm Dan}}$, in which case we must have $m  = 1$. Furthermore, by Danielewski's famous result, $X_{2,2,2} \times \Aff^{1} \isom D_m \times \Aff^{1}$ for every $m \geq 1$. This shows that $X_\ba \times \Aff^{1} \isom X' \times \Aff^{1}$ does not imply that $X' \isom X_\ba$ for some $\ba \in (\Nat_{\geq 2})^3$. If however, we impose the additional restriction that $X' = X_\bb$ for some $\bb \in (\Nat_{\geq 2})^3$, then these surfaces are cancellative. We obtain: 
\end{remark}

\begin{corollary}
    Let $\ba, \bb \in (\Nat_{\geq 2})^3$. If $X_\ba \times \Aff^1 \isom X_\bb \times \Aff^1$, then $X_\ba \isom X_\bb$ and $\ba \sim \bb$. 
\end{corollary}
\begin{proof}
    Assume $X_\ba \times \Aff^1 \isom X_\bb \times \Aff^1$. Then $X_\ba$ is non-rigid if and only if $X_\bb$ is non-rigid (by \cite[Theorem 2.24]{freudenburgBook}, for example). If $X_\ba$ is not rigid then $X_\bb$ is not rigid, and so $\ba, \bb \in S_{\text{\rm Dan}}$. If particular we may assume $\ba = (2,2,a_3)$ and $\bb = (2,2,b_3)$ by Example \ref{paramClosedExamples}. By Table \ref{invariants}, we have $a_3-1 = \rank(\Cl(X_\ba)) = \rank(\Cl(X_\ba \times \Aff^1)) = \rank(\Cl(X_\bb \times \Aff^1)) = \rank(\Cl(X_\bb)) = b_3-1$ and so $\ba \sim \bb$.
    
    If $X_\ba$ is rigid, then so is $X_\bb$ and by \cite[Lemma 2]{bandman2005nonstability}, $X_\ba = \ML(X_\ba \times \Aff^1) \isom \ML(X_\bb \times \Aff^1) = X_\bb$. So $X_\ba \isom X_\bb$ and by Theorem \ref{mainTheoremText}, $\ba \sim \bb$.   
\end{proof}

\bibliography{bibliography}
\bibliographystyle{alpha}

\end{document}